\documentclass[reqno,11pt]{amsart}
\newcommand{\be}{\begin{eqnarray}}
\newcommand{\ee}{\end{eqnarray}}
\newcommand{\beq}{\begin{equation}}
\newcommand{\eeq}{\end{equation}}
\newcommand{\ben}{\begin{eqnarray*}}
\newcommand{\een}{\end{eqnarray*}}

\newtheorem{theorem}{Theorem}

\newtheorem{lemma}[theorem]{Lemma}
\newtheorem*{maintheorem}{Theorem}
\newtheorem*{mainlemma}{Key Lemma}

\newtheorem{proposition}[theorem]{Proposition}
\newtheorem{remark}[theorem]{Remark}

{\catcode`\@=11\global\let\AddToReset=\@addtoreset
\AddToReset{equation}{section}

\AddToReset{theorem}{section}

\usepackage{graphicx}
\usepackage{psfrag}
\usepackage{epsfig,color}

\newcommand{\RR}{\mathbb R}
\newcommand{\inte}{\int_{\mathbb R^N}}

\begin{document}
\title[Uniqueness of the limit]{On the uniqueness of the limit for an asymptotically  autonomous semilinear  equation on $\mathbb R^N$}
      {   }\thanks{This research was supported by
        FONDECYT-1110074 for the first author, 
        FONDECYT-1110268 for the second autor and FONDECYT-11121125 for the  third author.}
\author{Carmen Cort\'azar}
\address{Departamento de Matem\'atica, Pontificia
        Universidad Cat\'olica de Chile,
        Casilla 306, Correo 22,
        Santiago, Chile.}
\email{\tt ccortaza@mat.puc.cl}
\author{Marta Garc\'{\i}a-Huidobro}
\address{Departamento de Matem\'atica, Pontificia
        Universidad Cat\'olica de Chile,
        Casilla 306, Correo 22,
        Santiago, Chile.}
\email{\tt mgarcia@mat.puc.cl}
\author{Pilar Herreros }
\address{Departamento de Matem\'atica, Pontificia
        Universidad Cat\'olica de Chile,
        Casilla 306, Correo 22,
        Santiago, Chile.}
\email{\tt pherrero@mat.puc.cl}


\begin{abstract}
We consider a parabolic equation of the form
\begin{equation*}
\begin{gathered}
u_t=\Delta u +f(u)+h(x,t),\quad (x,t)\in\mathbb R^N\times (0,\infty)\\
u(x,t)\ge 0\quad\mbox{for all }(x,t)\in\mathbb R^N\times (0,\infty).
\end{gathered}
\end{equation*}
where $f\in C^1(\mathbb R)$ is such that $f(0)=0$ and $f'(0)<0$ and $h$ is a suitable function on $\mathbb R^N\times (0,\infty)$. We show that under certain conditions, each globally defined and nonnegative bounded solution $u$ converges to a single steady state.
\end{abstract}

\maketitle

\section{Introduction and main results}
We consider non-negative, bounded and globally defined solutions of the parabolic problem
\begin{equation}\label{eq1}
\begin{gathered}
u_t=\Delta u +f(u)+h(x,t),\quad (x,t)\in\mathbb R^N\times (0,\infty)\\
u(x,0)=u_0(x)
\end{gathered}
\end{equation}
with $u_0\in H^1(\RR^N)$,
 $f\in C^1(\mathbb R)$ and $h$  a suitable function on $\mathbb R^N\times (0,\infty)$. By a bounded globally defined solution of \eqref{eq1} we mean a bounded function $u\in C([0,\infty),H^1(\RR^N))$ such that
\ben
\int_0^\infty\inte (\nabla u\cdot \nabla \varphi-f(u)\varphi-h\varphi-u\varphi_t)dx\ dt=\inte u_0(x)\varphi(x,0)dx
\een
for all $\varphi\in C^1_c([0,\infty),H^1(\RR^N))$.

 This problem was considered first in the autonomous case by Feireisl and Petzeltov\'a in \cite{fp}
for
$$f(s)=-s-\sum_{j=1}^mb_js^{r_j}+\sum_{j=1}^na_js^{p_j},$$
 with $a_j,\ b_j>0$ and $1<r_j<p_i\le N/(N-2)$, and later by Cort\'azar, del Pino and Elgueta
in \cite{cdpe}
for the particular nonlinearity
$$f(s)=s^p-s\quad 1<p<\frac{N+2}{N-2}\quad\mbox{
if $N\ge 3$, $1<p<\infty$ if $N=2$,}$$ and for compactly supported initial datum $u_0\in C(\RR^N)$.

The key common property of this type of nonlinearities, which leads to the convergence result, is that the stationary problem
\beq\label{stat0}
\begin{gathered}
\Delta w+f(w)=0,\quad \mbox{ in }\RR^N\\
 w(x)\to0\mbox{ as }|x|\to\infty
 \end{gathered}
\eeq
has, up to translations, a unique positive solution $w$ which is  symmetric with respect to some $x_0$, see for example \cite{gnn,ps1, sz}. Moreover, $w$ is  nondegenerate in the sense that the  linearized operator $\Delta+f'(w)$  on $L^2(\RR^N)$, restricted to $L^2_r(\RR^N)$ has a bounded inverse, or, in other words, its kernel is $N$ dimensional.
\bigskip

Later, Busca, Jendoubi, and  Pol\'a\v{c}ik in \cite{bjp}
extended these results for a  solution $u$ of \eqref{eq1} satisfying
 $$\leqno{(UD)}\qquad \quad
\lim\limits_{|x|\to\infty}\sup\limits_{t\in(0,\infty)}u(x,t)=0,$$
with $f\in C^1(\mathbb R)$, $f'(0)\not=0$,  and exponentially decreasing initial datum $u_0\in C(\RR^N)$.

They introduce the {\em first moments of energy}, which are constant in the $\omega$-limit set of $u$ as the usual energy is, allowing them to discriminate among the different translates of a stationary solution. They eliminate the possibility of a continuum of solutions by verifying that any $w\in\omega(u)$ has the property of {\em normal hiperbolicity}, meaning that the dimension of the kernel of  $\Delta+f'(w)$, (which by the symmetry of $w$ is  $N$ or   $N+1$), is the same as the dimension of the manifold formed by the set of all steady states in some neighborhood of $w$.

For the non-autonomous case we refer to a recent work of F\"oldes and Pol\'a\v{c}ik in \cite{pol1}
and the references therein, see for example \cite{chill1, chill2, chill3}. They assume $f\in C^1(\mathbb R)$, $f'(0)\not=0$, and under the minimal assumption that $h\in L^\infty(\RR^N\times(0,\infty))$ satisfies
$$\lim_{t\to\infty}||h(\cdot,t)||_\infty=0,$$
they prove a quasiconvergence result, \cite[Theorem 2.2]{pol1}, which essentially states that if $u$ is a non-negative, bounded and globally defined
solution of \eqref{eq1} satisfying $(UD)$, then either the $\omega$-limit set  $\omega(u)=\{0\}$ or it consists of ground states of \eqref{stat0}.

However, in order to prove their convergence result (\cite[Theorem 2.1]{pol1}), they impose a strong restriction on the decay of $h$, namely, they assume that
 there exist $\alpha\in(0,1)$, $\mu>0$ and $C^*>0$   such that $\tilde h(x,t):=e^{\mu t}h(x,t)$ satisfies
\begin{equation*}
(H)\quad
\begin{cases}
\mbox{either }||\tilde h||_{C^\alpha((0,\infty), L^\infty(\RR^N))}\le C^*,\\
\mbox{or }\quad ||\tilde h||_{L^\infty((0,\infty), C^\alpha(\RR^N))}\le C^*.
\end{cases}
\end{equation*}

This condition of  exponential decay on $h$ allows them to prove that all elements in $\omega(u)$ are symmetric with respect to the same center, and, thus, if $\omega(u)$ is not a single steady state, then some of its elements  must be contained on a normally hyperbolic manifold of steady states. They eliminate this possibility by viewing $u$ as a solution of an autonomous system to which they can apply a convergence result in \cite{pol2, hale} which also requires exponential decay in $h$.

Our result complements the result of \cite{pol1} in the sense that by imposing a stronger assumption on $f$, we may allow $h$ to have a much slower decay.
We follow the ideas  in \cite{cdpe}, from which we have kept the essential assumptions on $f$ that are needed for this approach:
We assume that
\begin{enumerate}
\item[(f1)] $f\in C^1(\mathbb R)$ is such that $f(0)=0$ and $f'(0)<0$,
\item[(f2)] Any nontrivial, nonnegative solution of the problem \eqref{stat0}
 is non-degenerate, that is, the linearized operator
\beq\label{linear}
\Delta +f'(w)\quad\mbox{on $L^2(\RR^N)$ with domain $W^{2,2}(\RR^N)$}
\eeq
restricted to $L^2_r(\RR^N)$ has a bounded inverse.
\end{enumerate}
As for the perturbation term, we assume that $h:\RR^N\times[0,\infty)\to\RR$ satisfies the following conditions:
\begin{enumerate}
\item[(h1)] $\lim_{t\to\infty}||h(\cdot,t)||_{L^\infty(\mathbb R^N)}=0$.
\end{enumerate}
We also assume that there exists a decreasing function $\hat h:[0,\infty)\to[0,\infty)$ such that
\begin{enumerate}
\item[(h2)] $||h(\cdot,t)||_{L^2(\mathbb R^N)}\le \hat h(t)$ for all $t\ge0$,
\item[(h3)] $\int^{\infty}\hat h^2(t)dt<\infty$,
\item[(h4)] $\displaystyle\lim_{t\to\infty}\frac{\hat h^2(t)}{\int_t^{\infty}\hat h^2(s)ds}=0$,
\item[(h5)] $\displaystyle\int_1^\infty\Bigl(\int_t^{\infty}\hat h^2(s)ds\Bigr)^{1/2}dt<\infty$.
\end{enumerate}

Observe that from the monotonicity of $\hat h$ and $(h3)$, we have that
\beq\label{h2}
\lim_{t\to\infty}\hat h(t)=0.
\eeq
\bigskip

Our main result is

\begin{maintheorem}
Let all the above assumptions on $f$ and $h$ hold. Let $u(x,t)$ be a non-negative, bounded, globally defined solution of \eqref{eq1} satisfying
\beq\label{c0}
\lim_{|x|\to\infty}\sup_{t\in(0,\infty)}u(x,t)=0.
\eeq
Then  there exists   a solution $w$ of
\eqref{stat0}
such that $u(x,t)\to w(x)$ uniformly as $t\to\infty$. Moreover, for any fixed $K>0$,
\beq\label{c00}
\lim_{t\to\infty}\int_0^K||u(\cdot,t+s)-w||^2_{H^1}ds= 0.
\eeq
\end{maintheorem}
\bigskip

Note that $\hat h(t)=t^\alpha$ obeys all conditions $(h1)$ through $(h5)$ for
$$ \alpha<-\frac{3}{2}.$$

We do not know if this is optimal, considering Example 2.3 built in \cite{foldes}.
Indeed, for a nonlinearity like $f(s)=s^p-s$ with $p>1$ and subcritical, it is known that the ground state solution $w$ of \eqref{stat0} is unique and symmetric with respect to some origin $x_0$. F\"oldes constructs  a bounded differentiable function $\eta:[0,\infty)\to\RR$, satisfying $|\eta'(t)|\le C/t$ for all $t>0$, and such that there are two sequences $s_n\to\infty$ and $t_n\to\infty$, with $\eta(s_n)=0$, $\eta(t_n)=1$ and
then he defines
$$u(x,t)=w(x_1+\eta(t),x_2,\ldots,x_N).$$
$u$ satisfies
$$u_t=\Delta u+f(u)+h(x,t),\quad h(x,t)=w_{x_1}(x_1+\eta(t),x_2,\ldots,x_N)\eta'(t),$$
hence $|h(x,t)|\le C_0/t$, and $w(x_1+s,x_2,\ldots,x_N)\in\omega(u)$ for all $s\in[0,1]$, implying that $\omega(u)$ is not a single steady state.

\begin{remark}\label{rem1}{\rm
By \cite[Theorem 2.2]{pol1},  under the assumptions of our theorem, any sequence $\{t_n\}$, $t_n\to\infty$, contains a subsequence, which we re-label the same, such that
$u(\cdot,t_n)$ converges uniformly either to $0$ or to a ground state solution $w$ of  \eqref{stat0}.
 }
\end{remark}
 Condition $(f2)$ and our assumptions on $h$
 allow the use of an iterative procedure induced by the Main Lemma (see Lemma \ref{paso1}) which roughly states that once $u$ becomes close to a steady state $w$ in an interval $[t^*,t^*+T]$,  at distance $d$, then a small translation of $w$ in an amount proportional to $d$ bounds the distance on the later time interval $[t^*+T,t^*+2T]$, by half of $d$ plus a constant times the $L^2$ norm of $\hat h$ in the interval $[t^*,\infty)$. All this   provided that $t^*$ sufficiently large and $T$ is larger than a certain constant. This implies, after iteration, that $u$ at later times never gets farther from $w$ than a constant times $d$. This idea was  used in \cite{mazzeo} in order to prove uniqueness of asymptotic profiles of solutions in the neighborhood of isolated singularities of an elliptic equation involving critical exponents. It was also used in \cite{cdpe} to treat the autonomous case.

Our paper is organized as follows. In section 2 we demonstrate our main result.  Its proof relies on a key lemma which is proved in section 3. Finally, in the appendix, we prove some  technical results that are used throughout the paper.

\section{Proof of the Main Theorem}
In this section we deduce our theorem from the following key lemma. The rest of the paper is devoted to the proof of this lemma.

 From now on we denote by $u(x,t)$ a fixed non-negative, bounded, globally defined solution of \eqref{eq1} satisfying \eqref{c0}, and
 $$||\cdot||_{L^2}=||\cdot||_{L^2(\mathbb R^N)}\quad\mbox{and}\quad||\cdot||_{H^1}=||\cdot||_{H^1(\mathbb R^N)}.$$

\begin{mainlemma}
Let $w$ be a non trivial nonnegative solution of \eqref{stat0}. Then there exist $T>1$, $t_0>0$ and $\eta_0>0$ such that for any $k\in\mathbb N$ and any $t^*>t_0$ with $\int_0^{T}||u(\cdot,s+t^*)-w(\cdot)||^2_{H^1}ds\le \eta_0^2$, it holds that
\be\label{kl}
\Bigl(\int_0^{T}||u(\cdot,s+t^*+kT)-w(\cdot)||^2_{H^1}ds\Bigr)^{1/2}\le \bar C\Bigl(\int_0^{T}||u(\cdot,s+t^*)-w(\cdot)||^2_{H^1}ds\Bigr)^{1/2}\nonumber\\+\bar C\int_{t^*-T}^\infty \Bigl(\int_t^\infty\hat h^2(s)ds\Bigr)^{1/2}dt
\ee
for some positive constant $\bar C=\bar C(T,w)$.

\end{mainlemma}

We will also need the following standard result concerning the functional
$$J(u)(t)=\frac{1}{2}\int_{\mathbb R^N}|\nabla u|^2dx-\int_{\mathbb R^N}F(u)dx,$$
where $F(u)=\int_0^uf(s)ds$, that we prove  for the sake of completeness.
\begin{lemma}\label{ln}
For any $t_2>t_1$,
\beq\label{jcomp}
J(u)(t_2)\le J(u)(t_1)+\frac{1}{4}\int_{t_1}^{t_2}\hat h^2(s)ds.
\eeq
\end{lemma}

\begin{proof}
The following formal calculations can be justified  approximating $h$ and $u_0$ by appropriate functions:
\be\label{dec}
\frac{d}{dt}J(u)(t)&=&\int_{\mathbb R^N}\nabla u\cdot\nabla u_t-f(u)u_tdx\nonumber\\
&=&-\int_{\mathbb R^N}(\Delta u +f(u)+h(x,t))u_tdx+\int_{\mathbb R^N}h(x,t)u_tdx\nonumber\\
&\le& -\int_{\mathbb R^N}(u_t)^2dx+\hat h(t)\Bigl(\int_{\mathbb R^N}(u_t)^2dx\Bigr)^{1/2}\nonumber\\
&\le& \frac{\hat h^2(t)}{4},
\ee
from where \eqref{jcomp} follows.
\end{proof}
\begin{proof}[Proof of the Main Theorem]
 We will first prove that $u(\cdot,t)$ is uniformly bounded in $H^1$. From \eqref{jcomp} in Lemma  \ref{ln},
 there exists $C_0>0$ such that
$$\frac{1}{2}\inte|\nabla u|^2dx-\inte F(u)dx\le C_0\quad\mbox{for all $t\ge t_0$}.$$

Next, from $(f1)$ and \eqref{c0}, there exists $R_0>0$ such that
$$F(u)\le \frac{f'(0)}{4}u^2\quad\mbox{for all $(x,t)$ with $|x|\ge R_0$,}$$
hence
\ben
\frac{1}{2}\inte|\nabla u|^2dx+\frac{|f'(0)|}{4}\inte u^2dx
&\le& C_0+\inte F(u)dx+\frac{|f'(0)|}{4}\inte u^2dx\\
&\le& C_0+\int_{B_{R_0}} F(u)dx+\frac{|f'(0)|}{4}\int_{B_{R_0}} u^2dx\\
&\le& C_0+|B_{R_0}|\Bigl(\sup_{s\in[0,M]}F(s)+M^2\frac{|f'(0)|}{4}\Bigr)
\een
and thus $||u(\cdot,t)||_{H^1}$ is uniformly bounded.

From Remark \ref{rem1}, we may assume that there is a sequence $\{t_n\}$, $t_n\to\infty$, such that
$u(\cdot,t_n)$ converges uniformly  to a nonnegative solution $w$ of  \eqref{stat0}.  As $u(\cdot,t_n)$ is uniformly bounded in $H^1$, there is a subsequence, which we re-label the same, converging weakly in $H^1$ and, up to a subsequence (still re-labeled the same), strongly in $L^2$.  Now, $||u(\cdot,t_n)-w(\cdot)||^2_{L^2}$ as $n\to\infty$ implies that
$$\int_0^{T}||u(\cdot,s+t_n)-w(\cdot)||^2_{H^1}ds\to0\quad\mbox{as}\quad n\to\infty,$$
(see Proposition \ref{teo22}a)(ii) in the Appendix).
If $||u(\cdot,t)||_{L^\infty}\to0$ as $t\to\infty$, then
$$\int_0^{T}||u(\cdot,s+t)||^2_{H^1}ds\to0\quad\mbox{as}\quad t\to\infty. $$
If $||u(\cdot,t)||_{L^\infty}\not\to0$ as $t\to\infty$, then again
from Remark \ref{rem1}, we may assume that $w$ is a non trivial, nonnegative solution  of  \eqref{stat0}.

Let now $\varepsilon>0$ and let $T$, $\eta_0$ and $t_0$ be as in the Key Lemma.
Then there exists $n_0$ such that $t_{n_0}\ge t_0$,
$$\int_0^{T}||u(\cdot,s+t_{n_0})-w(\cdot)||^2_{H^1}ds<\min\{\eta_0^2,\varepsilon\},$$
$$\int_{t_{n_0}}^\infty\hat h^2(s)ds<\varepsilon\quad\mbox{and}\quad\int_{t_{n_0}-T}^\infty \Bigl(\int_t^\infty\hat h^2(s)ds\Bigr)^{1/2}dt<\sqrt{\varepsilon}.$$
Hence, from the Key Lemma, for any $k\in\mathbb N$ we have
 \ben
\int_0^{T}||u(\cdot,s+t_{n_0}+kT)-w(\cdot)||^2_{H^1}ds\le 2\bar C\varepsilon.
\een
Let $t\ge t_{n_0}$. Then for some $k\in\mathbb N$, $t_{n_0}+kT\le t\le t_{n_0}+(k+1)T$, and thus $t=t_{n_0}+kT+\tau$ with $\tau\in[0,T]$.
Since there exists $C=C(T)$ such that
\ben
\int_{0}^{T}\!\!||u(\cdot,s+t_{n_0}+\tau+kT)-w||^2_{H^1}ds
\le C\Bigl(\int_{0}^{T}\!\!||u(\cdot,s+t_{n_0}+kT)-w||^2_{H^1}ds+\int_{t_{n_0}}^\infty \!\!\hat h^2(s)ds\Bigr)
\een
(see Proposition \ref{teo21} in the Appendix),  we obtain that there exists $C_0>0$ such that
\ben
\int_{0}^{T}||u(\cdot,s+t_{n_0}+\tau+kT)-w||^2_{H^1}ds
\le C_0\varepsilon,
\een
and thus for any $t\ge t_{n_0}$
$$\int_{0}^{K}||u(\cdot,s+t)-w||^2_{H^1}ds\le \Bigl(\Bigl[\frac{K}{T}\Bigr]+1\Bigr)C_0\varepsilon,$$
proving \eqref{c00}.

Finally, we observe that as each sequence $\{t_n\}$ has a subsequence (renamed the same) such that $u(x,t_n)$ converges uniformly to some solution $\tilde w$ of \eqref{stat0}, by the previous argument, we have that
\ben
\lim_{t\to\infty}\int_0^K||u(\cdot,s+t)-\tilde w||^2_{H^1}ds= 0,
\een
implying that $w=\tilde w$ and thus the uniform convergence follows.
\end{proof}

\section{Proof of  the Key Lemma }\label{3}
This section contains very heavy calculations, so in order to simplify the notation, we set
\beq\label{defeta}\eta(y,t):=\Bigl(\int_0^{T}||u(\cdot,s+t)-w(\cdot+y)||^2_{H^1}ds\Bigr)^{1/2},
\eeq
where $w$ is a fixed nontrivial, nonnegative solution of \eqref{stat0}. The proof of the Key Lemma will follow by induction from the following crucial result.

\begin{lemma}{\rm\bf(Main Lemma)}\label{paso1}
There exist $D>0$, $T>1$, $t_0>0$, $\eta_0>0$ and $A>0$ so that for all $(y,t)$ with $|y|\le1$, $t\ge t_0$, $\eta(y,t)\le \eta_0$,
there exists   $z\in\mathbb R^N$, with $|z|\le D\eta(y,t)$,   such that
\ben
\eta^2(z+y,t+T)\le \frac{1}{4}\eta^2(y,t)+A^2\int_{t}^\infty \hat h^2(s)ds.
\een
\medskip
\end{lemma}

 We will prove this lemma in two steps.

\noindent{\bf Claim 1.}
There exist $D>0$, $T>1$, $t_0$, $\eta_0$ so that for all $(y,t)$ with $|y|\le1$, $t\ge t_0$, $\eta(y,t)\le \eta_0$, and
$$\eta^2(y,t)\ge \int_t^\infty \hat h^2(s)ds,$$
there exists   $z\in\mathbb R^N$, with $|z|\le D\eta(y,t)$,  such that
\ben
\eta^2(z+y,t+T)\le \frac{1}{4}\eta^2(y,t).\een
\medskip

\noindent{\bf Claim 2.}

Let $T>1$ be as in Claim 1 and let $y\in\RR^N$. If
$$\eta^2(y,t)\le \int_t^\infty \hat h^2(s)ds,$$
then there exists a positive constant $A$ such that
$$\eta^2(y,t+T)\le A^2\int_{t}^\infty \hat h^2(s)ds.$$

\begin{proof}[Proof of Claim 1.]
We note that Claim 1 is equivalent to

\noindent{\bf Claim 1'.} There exist $D>0$, $T>1$  such that for every sequence $(y_n,t_n)$, with
\beq\label{caso2}
|y_n|\le1,\quad t_n\to\infty,\quad\eta(y_n,t_n)\to0\quad\mbox{and}\quad \eta^2(y_n,t_n)\ge \int_{t_n}^\infty \hat h^2(s)ds,
 \eeq
 there exist a subsequence $(y_{n'}, t_{n'})$ and a sequence $\{z_{n'}\}$, with $|z_{n'}|\le D\eta(y_{n'},t_{n'})$, satisfying
\beq\label{p3}
\eta^2(z_{n'}+y_{n'},t_{n'}+T)\le \frac{1}{4}\eta^2(y_{n'},t_{n'}).\eeq
The proof of this Claim will follow after several lemmas.

Let $T>1$ be a constant to be fixed later, and let $(y_n,t_n)$ be as above. By taking a subsequence if necessary, we may assume that $y_n\to y_0$ as $n\to\infty$. Also, by Remark \ref{rem1}, we may assume that $u(x,t_n)$ converges uniformly to some solution $\tilde w$ of \eqref{stat0}. This implies that  for any $K>1$ we have
\beq\label{cc}\lim_{n\to\infty}||u(\cdot,t+t_n)-\tilde w||_{L^\infty(\mathbb R^N\times[0,K])}=0,
 \eeq
 (see Proposition \ref{teo22}$b)$).
On the other hand, $w(\cdot+y_n)$ converges to $w(\cdot+y_0):=w_0(\cdot)$ in $H^1$, (see Proposition \ref{teo23}), hence  as $\eta(y_n,t_n)\to0$, we have that
\beq\label{c77}
\lim_{n\to\infty}\int_0^K||u(\cdot,t+t_n)-w_0||^2_{H^1}dt=0,
\eeq
and therefore $\tilde w=w_0$, and from \eqref{cc},
\beq\label{c7}
\lim_{n\to\infty}||u(\cdot,t+t_n)-w_0||_{L^\infty(\mathbb R^N\times[0,K])}=0.
\eeq
\bigskip

Set
$$\phi_n(x,t)=\frac{u(x,t+t_n)-w(x+y_n)}{\eta_n},$$
where $\eta$ is defined in \eqref{defeta} and $\eta_n=\eta(y_n,t_n)$.

We will show that $\phi_n$ converges to a solution $\phi$ of $\phi_t=\Delta \phi+f'(w_0)\phi$, and then use the properties of $\phi$ to obtain results concerning $\eta$ at later times.

Since
$\int_0^T||\phi_n(\cdot,t)||^2_{H^1}dt=1$,  there exists $r_n\in[1/2,1]$ such that $||\phi_n(\cdot,r_n)||_{H^1}\le 2$, and thus there is a subsequence  of $\{r_n\}$, still denoted the same, with
\beq\label{r0}
r_{n}\to r_0,\quad\mbox{ s.t. $\phi_{n}(\cdot,r_{n})$ converges weakly in $H^1$ to some function $\phi_0$}.
\eeq
Therefore, up to subsequences, we may assume that $\phi_{n}(\cdot,r_n)$ converges to $\phi_0$ in $L^2$.

\begin{lemma}\label{p5}
For each $K>1$, the integral $\int_0^K||\phi_n(\cdot,t)||^2_{H^1}dt$ is uniformly bounded in $n$.

\end{lemma}
\begin{proof}
By Lemma \ref{lemaprevio} (ii) with $t_1=r_n+t_n$ and $t_2=K+t_n$, we have that
$$\int_{r_n}^K||\phi_n(\cdot,s)||^2_{H^1}ds\le C_1(\bar M)e^{2\bar MK}\Bigl(||\phi_n(\cdot,r_n)||^2_{H^1}+\frac{1}{\eta_n^2}\int_{r_n+t_n}^{K+t_n}\hat h^2(s)ds\Bigr),$$
where $\bar M=\max_{u\in[0,M]}|f'(u)|+1$, and $M$ is a bound for $u$.
 By construction, $||\phi_n(\cdot,r_n)||^2_{H^1}\le 4$,  from the third condition in \eqref{caso2},
$$\frac{1}{\eta_n^2}\int_{r_n+t_n}^{K+t_n}\hat h^2(s)ds\le 1,$$
and, as $r_n\le 1<T$,
$$\int_0^{r_n}||\phi_n(\cdot,s)||^2_{H^1}ds\le \int_0^{T}||\phi_n(\cdot,s)||^2_{H^1}ds=1,
$$
hence
\ben
\int_0^K||\phi_n(\cdot,s)||^2_{H^1}ds\le5 C_1(\bar M)e^{2\bar MK}+1.
\een

\end{proof}

\begin{lemma}\label{convdebil}
For each $K>1$, there is a subsequence of $\{\phi_n\}$, still denoted the same, which converges weakly to some $\phi$ in $L^2([0,K),H^1(\mathbb R^N))$  that is
$$\int_0^K\int_{\mathbb R^N}(\nabla(\phi_n-\phi)\cdot\nabla\varphi+(\phi_n-\phi)\varphi)dxds\to 0\quad\mbox{as $n\to\infty$}$$
for all $\varphi\in C^\infty_0(\mathbb R^N\times [0,K))$. Moreover, $\phi$ is a solution of
\beq\label{eqphi}
\begin{gathered}
\phi_t=\Delta \phi+f'(w_0)\phi\\
\phi(x,r_0)=\phi_0(x)
\end{gathered}
\eeq
where  $r_0$ is as in \eqref{r0}.
\end{lemma}
\begin{proof}
The weak convergence follows from lemma \ref{p5}. We show next that $\phi$ satisfies \eqref{eqphi}. As
\be\label{16}
(\phi_n)_t&=&\Delta\phi_n+\frac{f(u(x,t+t_n))-f(w(x+y_n))}{\eta_n}+\frac{1}{\eta_n}h(x,t+t_n)\nonumber\\
&=&\Delta\phi_n+f'(\bar u_n)\phi_n+\frac{1}{\eta_n}h(x,t+t_n)
\ee
for some $\bar u_n$ between $u(x,t+t_n)$ and $w(x+y_n)$, by multiplying this equation by $\varphi\in C^\infty_0(\mathbb R^N\times [0,K))$ and  integrating over
$\mathbb R^N\times[r_n,K)$ we find that
\ben
\int_{r_n}^K\int_{\mathbb R^N}(\phi_n\varphi_t-\nabla\phi_n\cdot\nabla\varphi+f'(\bar u_n)\phi_n\varphi+\frac{1}{\eta_n}h(x,t+t_n)\varphi)dxdt\nonumber\\
=\int_{\mathbb R^N}\phi_n(x,r_n)\varphi(x,r_n)dx.
\een
Then,
\be\label{18}
\int_{r_0}^K\int_{\mathbb R^N}(\phi_n\varphi_t-\nabla\phi_n\cdot\nabla\varphi+f'( w_0)\phi_n\varphi)dxdt
-\int_{\mathbb R^N}\phi_0(x)\varphi(x,r_0)dx\qquad\qquad\qquad\nonumber\\
=\int_{r_0}^{r_n}\!\!\!\int_{\mathbb R^N}(\phi_n\varphi_t-\nabla\phi_n\cdot\nabla\varphi+f'(\bar u_n)\phi_n\varphi)dxdt
-\int_{\mathbb R^N}\!\!\!(\phi_0(x)\varphi(x,r_0)-\phi_n(x,r_n)\varphi(x,r_n))dx\nonumber\\
-\int_{r_0}^K\int_{\mathbb R^N}(f'(\bar u_n)-f'(w_0))\phi_n\varphi-\frac{1}{\eta_n}\int_{r_n}^K\int_{\mathbb R^N}h(x,t+t_n)\varphi dxdt\qquad\qquad
\ee

We show next that the right-hand side of \eqref{18} tends to 0 as $n\to\infty$: From lemma \ref{p5}, the first term satisfies
\ben
\Bigm|\int_{r_0}^{r_n}\int_{\mathbb R^N}(\phi_n\varphi_t-\nabla\phi_n\cdot\nabla\varphi+f'(\bar u_n)\phi_n\varphi)dxdt\Bigm|
\!\!\!&\le&\!\!\!|r_0-r_n|^{1/2}\tilde C(\varphi)\bar M\Bigl(\int_0^1||\phi_n(\cdot,s)||^2_{H^1}ds\Bigr)^{1/2}\nonumber\\
\!\!\!&\le&\!\!\!|r_0-r_n|^{1/2} C(\varphi),
\een
the second term tends to 0 because $\phi_n(x,r_n)\to \phi_0(x)$ in $L^2$,
\ben
\Bigm|\int_{r_0}^K\int_{\mathbb R^N}(f'(\bar u_n)-f'(w_0))\phi_n\varphi\Bigm|&\le&C(\varphi)\int_{r_0}^K\Bigl(\int_{\mathbb R^N}|f'(\bar u_n)-f'(w_0)|^2\phi_n^2dx\Bigr)^{1/2}dt\\
&\le&C(\varphi)||f'(\bar u_n)-f'(w_0)||_{L^\infty(\mathbb R^N\times[0,K])}\int_{r_0}^K||\phi_n(\cdot,t)||^2_{H^1}dt\nonumber
\een
which tends to $0$ by Lemma \ref{p5} and \eqref{c7}. Finally, by \eqref{h2} and \eqref{caso2}, the last term satisfies
$$\Bigm|\frac{1}{\eta_n}\int_{r_n}^K\int_{\mathbb R^N}h(x,t+t_n)\varphi dxdt\Bigm|\le \frac{1}{\eta_n}\hat h(t_n)K\int_0^K||\varphi(\cdot,s)||_{L^2}ds
\le C(K,\varphi)\frac{\hat h(t_n)}{(\int_{t_n}^\infty\hat h^2(s)ds)^{1/2}}$$
and tends to 0 by $(h4)$.
Therefore, by the weak convergence of $\phi_n$ to $\phi$ we obtain that
$$\int_{r_0}^K\int_{\mathbb R^N}(\phi\varphi_t-\nabla\phi\cdot\nabla\varphi+f'( w_0)\phi\varphi)dxdt
-\int_{\mathbb R^N}\phi_0(x)\varphi(x,r_0)dx=0$$
and thus $\phi$ is a weak solution of \eqref{eqphi}. It follows from standard linear parabolic theory that the weak-$H^1$ solution $\phi$ is actually a classical solution of class $C^{2,1}$ and furthermore, $\phi\in C((0,\infty), L^2(\mathbb R^N))$.
\end{proof}
The following  lemma concerns the strong convergence of $\phi_n$.
\begin{lemma}\label{ln2}
For any $K>1$ we have
\begin{enumerate}
\item[a)]
\ben
\lim_{n\to\infty}(\sup_{t\in[1,K]}||\phi_n(\cdot,t)-\phi(\cdot,t)||_{L^2})=0.
\een
\item[b)]
\ben\lim_{n\to\infty}\int_1^K||\phi_n(\cdot,t)-\phi(\cdot,t)||^2_{H^1}dt=0.
\een
\end{enumerate}
\end{lemma}

\begin{proof}
Let $\theta_n(x,t):=\phi_n(x,t)-\phi(x,t)$.
From   \eqref{eqphi} in Lemma \ref{convdebil} and \eqref{16}, $\theta_n$ satisfies
\ben
(\theta_n)_t=\Delta\theta_n+(f'(\bar u_n)-f'(w_0))\phi_n+f'(w_0)\theta_n+\frac{1}{\eta_n}h(x,t+t_n).
\een
Multiplying by $\theta_n$ and integrating we obtain
\be\label{theta_n}
 \frac{1}{2}(||\theta_n||^2_{L^2})_t&\le&-||\nabla\theta_n||^2_{L^2}+ ||f'(w_0)||_\infty||\theta_n||^2_{L^2}\nonumber\\
 &+& \!\!\! ||(f'(\bar u_n)-f'(w_0))||_{L^\infty(\mathbb R^N\times[0,K])} ||\phi_n||_{L^2}||\theta_n||_{L^2}+\frac{1}{\eta_n}\hat h(t_n)||\theta_n||_{L^2},
 \ee
hence also
$$ \frac{1}{2}(||\theta_n||^2_{L^2})_t\le C||\theta_n||^2_{L^2}+ ||(f'(\bar u_n)-f'(w_0))||_{L^\infty(\mathbb R^N\times[0,K])} ||\phi_n||_{L^2}||\theta_n||_{L^2}+\frac{1}{\eta_n}\hat h(t_n)||\theta_n||_{L^2}.$$
Dividing both sides by $||\theta_n||_{L^2}$,
 multiplying by $e^{-Ct}$ and integrating over $[r_n,t]$, we find that
\ben
||\theta_n(\cdot,t)||_{L^2}\le C||\theta_n(\cdot,r_n)||_{L^2}+C(||(f'(\bar u_n)-f'(w_0))||_{L^\infty(\mathbb R^N\times[0,K])}+\frac{1}{\eta_n}\hat h(t_n)
\een
and thus Part a) follows  by the continuity of $f'$, the third in \eqref{caso2} and  $(h4)$.

Next, from \eqref{theta_n}
\ben
\frac{1}{2}||\theta_n(\cdot,K)||^2_{L^2}+\int_1^K||\theta_n(\cdot,t)||^2_{H^1}dt\le \frac{1}{2}||\theta_n(\cdot,1)||^2_{L^2}+C\int_1^K||\theta_n(\cdot,s)||^2_{L^2}ds\\
+C\Bigl(\int_1^K||\theta_n(\cdot,s)||^2_{L^2}ds\Bigr)^{1/2}\Bigl(\int_1^K||\phi_n(\cdot,s)||^2_{H^1}ds\Bigr)^{1/2}+\frac{1}{\eta_n}\hat h(t_n)\int_1^K||\theta_n(\cdot,s)||^2_{L^2}ds,
\een
hence  from Lemma \ref{p5}
\ben
\int_1^K||\theta_n(\cdot,t)||^2_{H^1}dt\le
C(\sup_{t\in[1,K]}||\theta_n(\cdot,t)||^2_{L^2}+\sup_{t\in[1,K]}||\theta_n(\cdot,t)||_{L^2})
\een
 and the result follows from Part a).
\end{proof}

We study next the solution $\phi$ defined above. Let us consider the eigenvalue problem for \eqref{linear}, namely
$$\Delta\psi +f'(w_0)\psi=\lambda\psi,\quad  \psi\in L^2(\RR^N),\quad \psi\to0\mbox{ as }|x|\to\infty.$$
From $(f1)$, and using the fact that any solution of \eqref{stat0} is radially symmetric and decays exponentially, see for example \cite{bl1, ps1}, this problem has a finite number of eigenvalues in $[\frac{f'(0)}{2},\infty)$ with corresponding finite dimensional eigenspaces, see \cite{bsbook}. We denote by ${\lambda_i}$, $i=1,\ldots,q$, the positive eigenvalues counted with multiplicity.   By $(f2)$, the eigenspace $E_{0}$ corresponding to the (isolated) eigenvalue $0$ is $N$-dimensional, and spanned by $\{\frac{\partial  w_0}{\partial x_i}, i=1,\ldots, N\}$.   Following \cite{cdpe}, for $1/2\le r_0\le1$ as in \eqref{r0} we
decompose
\beq\label{phit0decomp}\phi_0(x)=\phi(x,r_0)=\sum_{i=1}^{q}B_ie^{\lambda_i r_0}\psi_i(x)+\sum_{i=1}^N C_i\frac{\partial w_0}{\partial x_i}(x)+\tilde\theta(x,r_0),
\eeq
with all terms in the right hand side mutually orthogonal in $L^2$, and set
\beq\label{deftildetheta}\tilde\theta(x,t):=\phi(x,t)-\sum_{i=1}^{q}B_ie^{\lambda_i t}\psi_i(x)-\sum_{i=1}^N C_i\frac{\partial w_0}{\partial x_i}(x).
\eeq
Then $\tilde\theta$ satisfies
$$\tilde\theta_t=\Delta\tilde\theta+f'(w_0)\tilde\theta\quad\mbox{in }\RR^n\times(0,\infty).$$
We have
\begin{lemma}\label{T}
 There exist $\alpha>0$ and $T_0>0$ depending only on $f$ and $N$,   such that for any $T\ge T_0$,
\beq\label{cotatildetheta}
\int_T^{2T}||\tilde\theta(\cdot,t)||_{H^1}^2dt\le e^{-\alpha T}.
\eeq
\end{lemma}
\begin{proof}
The proof of \eqref{cotatildetheta} follows exactly the same lines as the one given in the proof of \cite[Proposition 4.1]{cdpe} so we omit it. The only difference here is that $\alpha:=-\tilde\lambda/2$, where
\ben
-\tilde\lambda:&=&\inf\Bigl\{\inte(|\nabla \psi|^2-f'(w_0)\psi^2)dx\ |\ \mbox{$\psi$ is orthogonal (in $L^2$) to $E_{\lambda_i}$,}\nonumber\\
&&\qquad\qquad\mbox{ $i=0,\ldots,q$, and $\inte\psi^2dx=1$}\Bigr\}.
\een

\end{proof}

\begin{lemma}\label{B}
The coefficients $B_i$ in \eqref{phit0decomp} are zero for all $i=1,\ldots,q$, hence
\ben\phi(x,t)=\sum_{i=1}^N C_i\frac{\partial w_0}{\partial x_i}(x)+\tilde\theta(x,t).
\een

\end{lemma}
\begin{proof}

From  Lemma \ref{ln}, for any $0<\tau<t_n$,
\ben
J(u)(t_n+t)-J(u)(\tau+t)\le\frac{1}{4}\int^{\infty}_{\tau+t}\hat h^2(s)ds,
\een
hence, integrating over $[1,t_0]$ and letting  $n\to\infty$, by \eqref{c77} and \eqref{c7}, we get
\beq\label{n10}\int_1^{t_0}(J(u)(\tau+t)-J(w_0))dt\ge- \frac{1}{4}\int_1^{t_0}\int_{\tau+t}^\infty\hat h^2(s)dsdt.\eeq

We denote by $\langle\cdot,\cdot\rangle$ the usual inner product in $H^1(\mathbb R^N)$. As $J'(w_0)=0$, by using the Taylor's expansion of $J$ around $w_0$ we obtain
\be\label{519}
\int_1^{t_0}(J(u)(t_n+t)-J(w_0))dt=\frac{\eta_n^2}{2}\int_1^{t_0}\langle J''(w_0)\phi_n,\phi_n\rangle\ dt\nonumber\\
+\eta_n^2\int_1^{t_0}\int_0^1(1-\mu)\langle (J''(w_0+\mu(u_n-w_0))-J''(w_0))\phi_n,\phi_n\rangle d\mu\ dt.
\ee
As $J''$ is continuous and $u_n$ converges uniformly, we get from Lemma \ref{p5} that
$$\int_1^{t_0}\int_0^1(1-\mu)\langle (J''(w_0+\mu(u_n-w_0))-J''(w_0))\phi_n,\phi_n\rangle d\mu\ dt\le 1/2$$
for $n\ge n_0(t_0)$.
As $\theta_n=\phi_n-\phi$, we can write
$$\int_1^{t_0}\langle J''(w_0)\phi_n,\phi_n\rangle\ dt=\int_1^{t_0}(\langle J''(w_0)\phi,\phi\rangle+2\langle J''(w_0)\phi,\theta_n\rangle+\langle J''(w_0)\theta_n,\theta_n\rangle)dt.
$$
From $b)$ in Lemma \ref{ln2}, for $n\ge n_1(t_0)\ge n_0(t_0)$,
\be\label{nn}
\int_1^{t_0}\langle J''(w_0)\theta_n,\theta_n\rangle dt=\int_1^{t_0}\int_{\mathbb R^N}(|\nabla \theta_n|^2-f'(w_0)\theta_n^2)dxdt\nonumber\\
\le C\int_1^{t_0}||\theta_n||^2_{H^1}dt\le 1
\ee
 and
\be\label{0n}
\int_1^{t_0}\langle J''(w_0)\phi,\theta_n\rangle dt=\int_1^{t_0}\int_{\mathbb R^N}(\nabla\phi\cdot\nabla\theta_n-f'(w_0)\phi\theta_n) dx dt\nonumber\\
\le \Bigl(\int_1^{t_0}||\phi||^2_{H^1}dt\Bigr)^{1/2}\Bigl(\int_1^{t_0}||\theta_n||^2_{H^1}dt\Bigr)^{1/2}
\le 1.
\ee
Finally, using the decomposition given in \eqref{deftildetheta}, namely
$$\phi(x,t)=\sum_{i=1}^{q}B_ie^{\lambda_i t}\psi_i(x)+\sum_{i=1}^N C_i\frac{\partial w_0}{\partial x_i}(x)+\tilde\theta(x,t),
$$
we have
\ben
\int_1^{t_0}\langle J''(w_0)\phi,\phi\rangle dt=
\sum_{i,j=1}^{q}\langle J''(w_0)\psi_i,\psi_j\rangle\int_1^{t_0}B_iB_je^{(\lambda_i+\lambda_j)t}dt+\qquad\qquad\\
2\sum_{i=1}^{q}\int_1^{t_0}\!\!\!B_ie^{\lambda_i t}\langle J''(w_0)\psi_i,\sum_{i=1}^NC_i\frac{\partial w_0}{\partial x_i}+\tilde\theta\rangle dt
+\!\!\!\int_1^{t_0}\!\!\!\langle J''(w_0)(\sum_{i=1}^NC_i\frac{\partial w_0}{\partial x_i}+\tilde\theta),\sum_{i=1}^NC_i\frac{\partial w_0}{\partial x_i}+\tilde\theta\rangle dt
\een
\ben
=-\sum_{i=1}^{q}\!B_i^2\inte\!\!\!\psi_i^2dx\frac{e^{2\lambda_it_0}-e^{2\lambda_i}}{2}+\!\!\!\sum_{i,j=1,i\not=j}^{q}\!\!\!\underbrace{\langle J''(w_0)\psi_i,\psi_j\rangle}_{=0} B_iB_j\frac{e^{(\lambda_i+\lambda_j)t_0}-e^{(\lambda_i+\lambda_j)}}{\lambda_i+\lambda_j}\\
+2\sum_{i=1}^{q}\!\!\int_1^{t_0}\!\!\!B_ie^{\lambda_i t}\langle J''(w_0)\psi_i,\sum_{i=1}^NC_i\frac{\partial w_0}{\partial x_i}+\tilde\theta\rangle dt
+\!\!\!\int_1^{t_0}\!\!\!\langle J''(w_0)(\sum_{i=1}^NC_i\frac{\partial w_0}{\partial x_i}\!+\!\tilde\theta),\sum_{i=1}^NC_i\frac{\partial w_0}{\partial x_i}\!+\!\tilde\theta\rangle dt.
\een
Let $i_0$ be such that $\lambda_{i_0}=\max\{\lambda_i\ |\ B_i\not=0\}$. Then, from lemma \ref{T}, we can write
\beq\label{00}\int_1^{t_0}\langle J''(w_0)\phi,\phi\rangle dt\le -CB_{i_0}^2\frac{e^{2\lambda_{i_0}t_0}-e^{2\lambda_{i_0}}}{2}+Ce^{\lambda_{i_0}t_0}
\eeq
for some positive constant $C$ independent of $t_0$.
Hence for $n$ large enough,  we get from \eqref{n10}, \eqref{519}, \eqref{nn}, \eqref{0n}, and \eqref{00},
\ben
- \frac{1}{4}\int_1^{t_0}\int_{t_n+t}^\infty\hat h^2(s)dsdt\le \int_1^{t_0}(J(u)(t_n+t)-J(w_0))dt\qquad\qquad\\
\le
\frac{\eta_n^2}{2}\Bigl[-CB_{i_0}^2\frac{e^{2\lambda_{i_0} t_0}-e^{2\lambda_{i_0}}}{2}+Ce^{\lambda_{i_0}t_0}+3\Bigr]
\een
and thus
$$
- \frac{t_0}{4}\int_{t_n}^\infty\hat h^2(s)ds\le\frac{\eta_n^2}{2}\Bigl[-CB_{i_0}^2\frac{e^{2\lambda_{i_0} t_0}-e^{2\lambda_{i_0}}}{2}+Ce^{\lambda_{i_0}t_0}+3\Bigr].
$$
Recalling assumption
$$\eta^2(y_n,t_n)\ge \int_{t_n}^\infty \hat h^2(s)ds$$
in \eqref{caso2},
we have a contradiction if $t_0$ is large enough. Hence $B_{i}=0$ for all $i$.

\end{proof}
We are now in a position to conclude the proof of Claim 1'. Let $\vec C=(C_1,C_2,...,C_N)$, where the $C_i's$ are as defined in \eqref{phit0decomp}, $D=|\vec C|$,   $z_n=\eta_n\vec C$, $ T_0$ as given in Lemma \ref{T}, and $T\ge T_0$  such that $e^{-\alpha T}\le 1/8$ . Then by  Lemma \ref{B}
\ben
u(x,t+t_n)-w(x+y_n+z_n)=\eta_n(\theta_n+\phi)+w(x+y_n)-w(x+y_n+z_n)\\
=\eta_n\theta_n(x,t)+\eta_n\sum_{i=1}^NC_i\frac{\partial w_0}{\partial x_i}+\eta_n\tilde\theta(x,t)+
w(x+y_n)-w(x+y_n+z_n)\\
=\eta_n\theta_n(x,t)+\eta_n\tilde\theta(x,t)
-(w(x+y_n+z_n)-w(x+y_n)-\nabla w_0(x)\cdot z_n),
\een
hence by Proposition \ref{teo23},
\ben
\eta^2(y_n+z_n,t_n+T)&=&\int_T^{2T}||u(x\cdot,t+t_n)-w(\cdot+y_n+z_n)||_{H^1}^2dt\\
&\le&
\eta_n^2(\int_T^{2T}||\theta_n(\cdot,t)||_{H^1}^2dt+\int_T^{2T}||\tilde\theta(x,t)||_{H^1}^2dt+TO(\eta_n)).
\een
We now fix  $n_0=n_0(T)$ so that the first and third term in the parenthesis add to something less than  $1/8$ for $n\ge n_0$. Then from Lemma \ref{T} we have
\ben
\eta^2(y_n+z_n,t_n+T)&=&\int_T^{2T}||u(x\cdot,t+t_n)-w(\cdot+y_n+z_n)||_{H^1}^2dt
\le \frac{1}{4}\eta_n^2
\een
proving \eqref{p3} in Claim 1' and thus Claim 1 follows.

\end{proof}

\begin{proof}[Proof of Claim 2.]
From our assumption, there exists $t_0\in[t,t+T]$ such that
$$||u(\cdot,t_0)-w(\cdot+y)||^2_{H^1}\le \frac{1}{T}\int_t^\infty \hat h^2(s)ds.$$
From lemma \ref{lemaprevio}(ii),  with $t_1=t_0$ and $t_2=t+2T$, we find that
\ben
\int_{t+T}^{t+2T}||u(\cdot,s)-w(\cdot+y)||^2_{H^1}ds&\le& \int_{t_0}^{t+2T}||u(\cdot,s)-w(\cdot+y)||^2_{H^1}ds\\
&\le&
C_1(\bar M)e^{4\bar MT}\Bigl(\frac{1}{T}+1\Bigr)\int_t^\infty \hat h^2(\xi)d\xi,
\een
hence the result follows with $A^2=2C_1(\bar M)e^{4\bar MT}$.

\end{proof}

The following result will follow by induction from Lemma \ref{paso1}.

\begin{lemma}\label{lemapili}
There exist $D>0$, $T>1$, $\bar t>0$, $\bar\eta>0$ and $A>0$ such that for any $t^*>\bar t$,   with $\eta(0,t^*)\le \bar\eta$, there exist   $\{x_i\}\subset\mathbb R^N$, with $|x_i|\le D\eta(x_1+\cdots+x_{i-1},t^*+(i-1)T)$,    it holds that
\beq\label{p0}
\eta(x_1+\cdots+x_{k},t^*+kT)\le \frac{1}{2}\eta(x_1+\cdots+x_{k-1},t^*+(k-1)T)+A\Bigl(\int_{t^*+(k-1)T}^\infty\hat h^2(s)ds\Bigr)^{1/2}.\eeq

\end{lemma}

\begin{proof}
Let $T$, $t_0$, $\eta_0$, $A$ and $D$ be as in  Lemma \ref{paso1}. Let $\bar\eta$ and $\bar t$ be such that
\beq\label{bareta}
\bar \eta\in(0,\eta_0),\quad \bar\eta\le \frac{1}{4(D+1)}
\eeq
and
\beq\label{bart}
\bar t\ge t_0,\quad A^2\int_{\bar t}^\infty \hat h^2(s)ds<\bar\eta^2/4,\quad 2DA\int_{\bar t-T}^\infty\Bigl(\int_t^\infty\hat h^2(s)ds\Bigr)^{1/2}\ dt<1/2.
\eeq
Let $t^*\ge \bar t$. We will prove the existence of the sequence $\{x_i\}$ by induction by proving that for each $k\in\mathbb N$, there exists $x_k$ satisfying:
\beq
\begin{cases}
\mbox{\eqref{p0}},\\
|x_k|\le D\eta(x_1+\cdots+x_{k-1},t^*+(k-1)T),\\
|x_1+x_2+\cdots+x_k|\le 1\quad\mbox{and}\\  \eta(x_1+x_2+\cdots+x_k,t^*+kT)<\bar\eta.
\end{cases}\label{indass}
\eeq
For $k=1$, \eqref{indass} is true thanks to Lemma \ref{paso1} and the choice of the parameters. Assume now that \eqref{indass} is true for all positive integers less than or equal to $k$. Then, we can apply Lemma \eqref{paso1} to obtain the existence of $x_{k+1}$ such that
\eqref{p0} holds for $k+1$ and
$$|x_{k+1}|\le D\eta(x_1+\cdots+x_{k},t^*+kT).$$
By assumption we have that
$$ \sum_{i=1}^{k+1}|x_i|\le D\sum_{i=1}^{k+1}\eta(x_1+x_2+\cdots+x_{i-1},t^*+(i-1)T).$$
Using repeatedly the induction hypothesis, we get that for $i\ge 2$,
$$
\eta(x_1+x_2+\cdots+x_{i-1},t^*+(i-1)T)\le \frac{1}{2^{i-1}}\eta(0,t^*)+A\sum_{j=1}^{i-1}\frac{1}{2^{i-1-j}}\Bigl(\int_{t^*+(j-1)T}^\infty\hat h^2(s)ds\Bigr)^{1/2},$$
hence by adding we obtain
\be\label{c11}
\sum_{i=1}^{k+1}\eta(x_1+x_2+\cdots+x_{i-1},t^*+(i-1)T)\!\!\!&\le&\!\!\! 2\eta(0,t^*)+2A\sum_{j=1}^{k}\Bigl(\int_{t^*+(j-1)T}^\infty\hat h^2(s)ds\Bigr)^{1/2}\qquad\nonumber\\
\!\!\!&\le&\!\!\! 2\eta(0,t^*)+2A\int_{t^*-T}^\infty\Bigl(\int_{t}^\infty\hat h^2(s)ds\Bigr)^{1/2}\!\!\!dt,
\ee
and thus, by the choice of $\bar \eta$ and $\bar t$ in \eqref{bareta} and \eqref{bart}, we obtain
$$|x_1+x_2+\cdots+x_k+x_{k+1}|\le 1.$$
On the other hand, using \eqref{p0} for $k+1$, \eqref{indass}, and \eqref{bart}, we obtain
$$\eta(x_1+x_2+\cdots+x_{k+1},t^*+(k+1)T)\le \frac{1}{2}\bar\eta+A\Bigl(\int_{t^*}^\infty\hat h^2(s)ds\Bigr)^{1/2}\le \frac{1}{2}\bar\eta+A\Bigl(\int_{\bar t}^\infty\hat h^2(s)ds\Bigr)^{1/2}\le \bar\eta,$$
thus   Lemma \ref{lemapili} follows.
\end{proof}

We are now ready to prove the Key Lemma.

\begin{proof}[Proof of the Key Lemma]
Observe that in terms of $\eta$, inequality \eqref{kl} in the statement of this lemma reads as
 $$\eta(0,t^*+kT)\le \bar C\eta(0,t^*)+\bar C\int_{t^*-T}^\infty\Bigl(\int_{t}^\infty\hat h^2(s)ds\Bigr)^{1/2}dt.$$
 Let  $D$, $A$, $T$, $\bar t$, $\bar\eta$ and $\{x_i\}$ as  in  Lemma \ref{lemapili}. We have, for some positive constant $C_w$ depending only on $w$,
\be\label{p1}
|\eta(x_1+x_2+\cdots+x_k,t^*+kT)\!\!&\!\!-\!\!&\!\!\eta(0,t^*+kT)|\le C_w\sum_{i=1}^k|x_i|\nonumber\\
&\le& C_wD\sum_{i=1}^k\eta(x_1+x_2+\cdots+x_{i-1},t^*+(i-1)T)\nonumber
\ee
thus, from \eqref{c11},
\ben
\eta(0,t^*+kT)&\le& (1+C_wD)\sum_{i=1}^{k+1}\eta(x_1+x_2+\cdots+x_{i-1},t^*+(i-1)T)\\
&\le&(1+C_wD)(2\eta(0,t^*)+2A\int_{t^*-T}^\infty\Bigl(\int_{t}^\infty\hat h^2(s)ds\Bigr)^{1/2}dt),
\een
and thus the Key Lemma follows with
 $\bar C=2(1+C_wD)(1+A)$.
\end{proof}

\section{Appendix: Some technical results}\label{final}

We state and prove here some technical results that have been used in the previous sections.
In what follows, $\tilde w$ is any nonnegative solution of \eqref{stat0}. We start with the following lemma.
\begin{lemma}\label{lemaprevio}
If $t_2>t_1$, then
\begin{enumerate}
\item[(i)]
$$||u(\cdot,t_2)-\tilde w||_{L^2}\le e^{\bar M(t_2-t_1)}\Bigl(||u(\cdot,t_1)-\tilde w||_{L^2}+\int_{t_1}^{t_2}\hat h(s)ds\Bigr).$$
\item[(ii)]
\ben
\int_{t_1}^{t_2}||u(\cdot,s)-\tilde w||^2_{H^1}ds&\le& \frac{1}{2}||u(\cdot,t_1)-\tilde w||^2_{L^2}+\bar M\int_{t_1}^{t_2}||u(\cdot,s)-\tilde w||^2_{L^2}ds\\
&&+\Bigl(\int_{t_1}^{t_2}\hat h^2(s)ds\Bigr)^{1/2}(\Bigl(\int_{t_1}^{t_2}||u(\cdot,s)-\tilde w||^2_{L^2}ds\Bigr)^{1/2}\\
&\le& C_1(\bar M)e^{2\bar M(t_2-t_1)}\Bigl(||u(\cdot,t_1)-\tilde w||^2_{L^2}+ \int_{t_1}^{t_2}\hat h^2(s)ds\Bigr)
\een
where $\bar M=\max_{u\in[0,M]}|f'(u)|+1$, and $M$ is a bound for $u$.
\end{enumerate}
\end{lemma}

\begin{proof}
Let $M$ and $\bar M$ as above. Using the equations satisfied by $u$ and $\tilde w$,  we obtain
$$(u-\tilde w)_t=\Delta(u-\tilde w)+f(u)-f(\tilde w)+h(x,t).$$
Multiplying by $u-\tilde w$ and integrating over $\mathbb R^N$ we get, for some  $\bar u\in[0,M]$,
\ben
\frac{1}{2}\Bigl(||u(\cdot,t)-\tilde w||^2_{L^2}\Bigr)_t&=&-\int|\nabla (u-\tilde w)|^2dx+\int f'(\bar u)(u-\tilde w)^2dx+\int h(x,t)(u-\tilde w)dx\\
&\le& -||u-\tilde w||^2_{H^1}+\bar M||u-\tilde w||^2_{L^2}+\hat h(t)||u-\tilde w||_{L^2},
\een
hence
\beq\label{2}\frac{1}{2}\Bigl(||u(\cdot,t)-\tilde w||^2_{L^2}\Bigr)_t+||u-\tilde w||^2_{H^1}\le
\bar M||u-\tilde w||^2_{L^2}+\hat h(t)||u-\tilde w||_{L^2},
\eeq
hence dividing by $||u-\tilde w||_{L^2}$ we get
$$\bigl(||u(\cdot,t)-\tilde w||_{L^2}\bigr)_t\le \bar M||u-\tilde w||_{L^2}+\hat h(t),$$
and thus multiplying by $e^{-\bar Mt}$ and integrating over $[t_1,t_2]$ $(i)$ follows. Finally, integrating \eqref{2} over $[t_1,t_2]$ and using H\"older and part $(i)$, $(ii)$ follows.
\end{proof}

\begin{proposition}\label{teo21}
Let $K>1$. Then there exists $C=C(K)$ such that for any $t>0$ and any $\tau\in[0,K]$,
\ben
\int_{0}^{K}||u(\cdot,s+\tau+t)-\tilde w||^2_{H^1}ds\le C\int_{0}^{K}||u(\cdot,s+t)-\tilde w||^2_{H^1}ds+C\int_{t}^{t+2K}\hat h^2(s)ds.
\een
\end{proposition}
\begin{proof}
 Let $K>1$.
 There exists $t_0\in[0,K]$ such that
 \ben
 ||u(\cdot,t+t_0)-\tilde w||_{L^2}\le  ||u(\cdot,t+t_0)-\tilde w||_{H^1}&\le& \int_t^{t+K}\!\!\!||u(\cdot,s)-\tilde w||_{H^1}ds\\
 &\le& \sqrt{K}\Bigl(\int_t^{t+K}\!\!\!||u(\cdot,s)-\tilde w||^2_{H^1}ds\Bigr)^{1/2}.\een
 By Lemma \ref{lemaprevio}(ii), and using that $\hat h$ is decreasing we have
 \ben
\int_{t+t_0}^{t+t_0+2K}||u(\cdot,s)-\tilde w||^2_{H^1}ds\!\!\!&\le&\!\!\! C_1(\bar M)e^{4K\bar M}\Bigl(K\int_t^{t+K}||u(\cdot,s)-\tilde w||^2_{H^1}ds+ \int_{t+t_0}^{t+t_0+2K}\hat h^2(s)ds\Bigr)\\
\!\!\!&\le&\!\!\! C_1(\bar M)e^{4K\bar M}K\Bigl(\int_t^{t+K}||u(\cdot,s)-\tilde w||^2_{H^1}ds+ \int_{t}^{t+2K}\hat h^2(s)ds\Bigr).
\een
Let $\tau\in[0,K]$. As
$$ \int_{t+\tau}^{t+\tau+K}||u(\cdot,s)-\tilde w||^2_{H^1}ds\le \int_{t}^{t+K}||u(\cdot,s)-\tilde w||^2_{H^1}ds+\int_{t+t_0}^{t+t_0+2K}||u(\cdot,s)-\tilde w||^2_{H^1}ds,$$
we find that
$$ \int_{t+\tau}^{t+\tau+K}||u(\cdot,s)-\tilde w||^2_{H^1}ds\le
(C_1(\bar M)Ke^{4K\bar M}+1)\Bigl(\int_t^{t+K}||u(\cdot,s)-\tilde w||^2_{H^1}ds+ \int_{t}^{t+2K}\hat h^2(s)ds\Bigr),$$
that is,
$$\int_{0}^{K}||u(\cdot,s+\tau+t)-\tilde w||^2_{H^1}ds\le C\int_{0}^{K}||u(\cdot,s+t)-\tilde w||^2_{H^1}ds+C\int_{t}^{t+2K}\hat h^2(s)ds,$$
with $C=C_1(\bar M)Ke^{4K\bar M}+1$.
\end{proof}
\begin{proposition}\label{teo22}\mbox{ }Let $t_n\to\infty$ as $n\to\infty$ and $K>1$.
\begin{enumerate}
\item[a)]
If $||u(\cdot,t_n)-\tilde w||_{L^2}\to0$ as $n\to\infty$, then
\begin{enumerate}
\item[(i)]
$$\max_{s\in[0,K]}||u(\cdot,s+t_n)-\tilde w||_{L^2}\to 0\quad\mbox{as }n\to\infty,$$
\item[(ii)] $$\int_0^K||u(\cdot,s+t_n)-\tilde w||^2_{H^1}ds\to 0.$$
\end{enumerate}

\medskip
\item[b)] If $||u(\cdot,t_n)-\tilde w||_{L^\infty(\mathbb R^N)}\to0$ as $n\to\infty$, then
$$\lim_{n\to\infty}||u(x,t+t_n)-\tilde w||_{L^\infty(\mathbb R^N\times[0,K])}=0.$$
\end{enumerate}
\end{proposition}
\begin{proof}Let $t_n\to\infty$ as $n\to\infty$ and $K>1$.

 \noindent { Part $a)$}. By \eqref{h2}, $\hat h(t_n)\to0$. By lemma \ref{lemaprevio}$(i)$ with $t_1=t_n$ and $t_2=s\in[t_n,t_n+K]$, we have, by the monotonicity of $\hat h$,
\ben
||u(\cdot,s)-\tilde w||_{L^2}&\le& e^{\bar M(s-t_n)}\Bigl(||u(\cdot,t_n)-\tilde w||_{L^2}+\int_{t_n}^{s}\hat h(s)ds\Bigr)\\
&\le& e^{\bar MK}\Bigl(||u(\cdot,t_n)-\tilde w||_{L^2}+K\hat h(t_n)\Bigr),
\een
proving $a)(i)$. Similarly, using Lemma \ref{lemaprevio}$(ii)$, result $(ii)$ follows.
\smallskip

 \noindent { Part $b)$}: From standard parabolic estimates, it holds that
$$||u(\cdot, t+t_n)-\tilde w(x)||_{L^\infty(\mathbb R^N)}\le e^{\bar Mt}||u(\cdot, t_n)-\tilde w(x)||_{L^\infty(\mathbb R^N)}+e^{\bar Mt}\int_0^te^{-\bar Ms}||h(\cdot,t_n+s)||_{L^\infty(\mathbb R^N)}ds,$$
hence the result follows by $(h1)$.
\end{proof}

\begin{proposition}\label{teo23}
 If $y_n\to 0$ in $\RR^N$, then
\begin{enumerate}
\item[(i)]
$||\tilde w(\cdot+y_n)-\tilde w-\nabla \tilde w\cdot y_n||_{H^1}= O(|y_n|^{3/2})$, and hence
\item[(ii)] $\tilde w(\cdot+y_n)\to \tilde w$ in $H^1$ as $n\to\infty$.
\end{enumerate}
\end{proposition}

\begin{proof}
Differentiating equation \eqref{stat0} with respect to $x_i$, then multiplying by $\tilde w_{x_i}$  and integrating, we obtain that
$$\inte |\nabla \tilde w_{x_i}|^2dx=\Bigm|\inte f'(\tilde w)\tilde w_{x_i}^2ds\Bigm|\le \bar M||\tilde w||_{H^1}^2.$$
On the other hand, as
\ben
\tilde w_{x_i}(x+y_n)-\tilde w_{x_i}(x)-\nabla \tilde w_{x_i}(x)\cdot y_n=\int_0^1(\nabla \tilde w_{x_i}(x+sy_n)-\nabla \tilde w_{x_i}(x))\cdot y_n ds,
\een
we find that
$$\inte |\tilde w_{x_i}(x+y_n)-\tilde w_{x_i}(x)-\nabla \tilde w_{x_i}(x)\cdot y_n|^2dx\le |y_n|^2\inte\int_0^1|\nabla \tilde w_{x_i}(x+sy_n)-\nabla \tilde w_{x_i}(x)|^2ds\ dx$$
\be\label{cota2}
&=& |y_n|^2\int_0^1\inte \left(|\nabla \tilde w_{x_i}(x+sy_n)|^2+|\nabla \tilde w_{x_i}(x)|^2-2\nabla \tilde w_{x_i}(x+sy_n) \cdot \nabla \tilde w_{x_i}(x)\right)dx\ ds\nonumber\\
&=&2|y_n|^2\int_0^1\inte \left(|\nabla \tilde w_{x_i}(x)|^2-\nabla \tilde w_{x_i}(x+sy_n) \cdot \nabla \tilde w_{x_i}(x)\right)dx\ ds\nonumber\\
&=&-2|y_n|^2\int_0^1\inte f'(\tilde w)\tilde w_{x_i}(x)(\tilde w_{x_i}(x+sy_n)- \tilde w_{x_i}(x)) dx\ ds\nonumber\\
&=&-2|y_n|^2\inte f'(\tilde w)\tilde w_{x_i}(x)\int_0^1\Bigl(\int_0^s\nabla \tilde w_{x_i}(x+\tau y_n)\cdot y_n d\tau\Bigr)\ ds\ dx.
\ee
Finally, since
\ben
\inte \Bigl(\int_0^1 \Bigl(\int_0^s\nabla \tilde w_{x_i}(x+\tau y_n)\cdot y_n d\tau \Bigr)\ ds \Bigr)^2\ dx&\le& |y_n|^2\inte \Bigl(\int_0^1\int_0^s|\nabla \tilde w_{x_i}(x+\tau y_n)|^2d\tau\ ds \Bigr)dx\\
&\le& |y_n|^2\inte |\nabla \tilde w_{x_i}(x)|^2dx
\een
we obtain from  \eqref{cota2} that
\ben
\inte |\tilde w_{x_i}(x+y_n)-\tilde w_{x_i}(x)-\nabla \tilde w_{x_i}(x)\cdot y_n|^2dx
\le  2\bar M^2|y_n|^3||\tilde w||_{H^1}^2.
\een
Similarly, arguing with $\tilde w$ instead of $\tilde w_{x_i}$, and noting that by $(f1)$,  $||f(\tilde w)||_{L^2}\le \bar M ||\tilde w||_{L^2}$, we obtain
\ben
\inte |\tilde w(x+y_n)-\tilde w(x)-\nabla \tilde w(x)\cdot y_n|^2dx
\le  2\bar M^2|y_n|^3||\tilde w||_{H^1}^2
\een
proving the proposition.
\end{proof}

\end{document}